\documentclass{amsart}
\usepackage[latin1]{inputenc}
\usepackage[english]{babel}
\usepackage[T1]{fontenc}
\usepackage{amsmath}
\usepackage{amssymb}
\usepackage{amsthm}
\usepackage{array,booktabs}
\usepackage{graphicx}
\usepackage{newtxtext}
\usepackage{quoting}
\usepackage{enumerate}
\usepackage{mathtools}
\usepackage{csquotes}
\usepackage{cite}
\usepackage{tikz}
\usepackage{tikz-cd}
\usetikzlibrary{matrix,arrows,decorations.pathmorphing}
\usepackage{wrapfig}
\renewcommand{\Re}{\operatorname{Re}}

\def\R{\ensuremath\mathbb{R}}
\def\C{\ensuremath\mathbb{C}}
\def\Z{\ensuremath\mathbb{Z}}
\def\Q{\ensuremath\mathbb{Q}}

\def\H{\ensuremath\mathbb{H}}

\usepackage{hyperref}
\usepackage[final]{pdfpages}
\usepackage{verbatim}
\newtheorem{thm}{Theorem}[section]
\newtheorem{defi}[thm]{Definition}
\newtheorem{cor}[thm]{Corollary}
\newtheorem{lemma}[thm]{Lemma}

\theoremstyle{remark}
\newtheorem{remark}[thm]{Remark}

\def\0{\emptyset}

\def\SL{\hbox{\rm SL}}
\def\GL{\hbox{\rm GL}}
\numberwithin{equation}{section}
\begin{document}
\title{A note on additive twists, reciprocity laws and quantum modular forms}
\author{Asbjørn Christian Nordentoft}
\address{Department of Mathematical Sciences, University of Copenhagen}
\date{September, 2019}
\begin{abstract}
We prove that the central values of additive twists of a cuspidal $L$-function define a quantum modular form in the sense of Zagier, generalizing recent results of Bettin and Drappeau. From this we deduce a reciprocity law for the twisted first moment of multiplicative twists of cuspidal $L$-functions, similar to reciprocity laws discovered by Conrey for the twisted second moment of Dirichlet $L$-functions. Furthermore we give an interpretation of quantum modularity at infinity for additive twists of $L$-functions of weight 2 cusp forms in terms of the corresponding functional equations. 
\end{abstract}
\maketitle
\section{Introduction}
In an unpublished paper \cite[Theorem 10]{Conrey07} Conrey discovered a certain reciprocity law satisfied by the twisted second moment of Dirichlet $L$-functions. The reciprocity law relates the following two twisted moments;
\begin{align}\label{dirichlettwist} \sum_{\chi \text{ mod }q}|L(\chi,1/2)|^2 \chi(l)   \rightsquigarrow \sum_{\chi \text{ mod }l}|L(\chi,1/2)|^2 \chi(-q),\end{align}
for primes $q\neq l$. Conrey's results were then generalized and refined by Young \cite{Young11} and Bettin \cite{Bettin16}. In this paper we prove a reciprocity law for twists of $\GL _2$-cusp forms, which in the simplest case relates the following two twisted first moments;
\begin{align} \label{reciprocity}\sum_{\substack{\chi \text{ mod }q,\\ \chi \text{ primitive}}}  \tau(\overline{\chi})L(f\otimes \chi,k/2) \chi(l) \rightsquigarrow  \sum_{\substack{\chi \text{ mod }l,\\ \chi \text{ primitive}}} \tau(\overline{\chi})L(f\otimes \chi,k/2)\chi(-q),\end{align}
where $f\in \mathcal{S}_k(\Gamma_0(1))$ is a cusp form of weight $k$ and level $1$ and $q\neq l$ are primes. The exact moments involved for general level $N$ and general $q,l$ are more involved (see Theorem \ref{recipr}). \\

Our proof uses an interpretation of the twisted moments (\ref{reciprocity}) in terms of additive twists of $L$-functions. The additive twists associated to a weight $k$ cusp form $f\in \mathcal{S}_k(\Gamma_0(N))$ are defined as
$$ L(f\otimes e(a/c),s):= \sum_{n\geq 1} \frac{a_f(n)e(na/c)}{n^s}, $$
for $\Re s>(k+1)/2$, where $a_f(n)$ denote the Fourier coefficients of $f$, $a/c\in \Q$ and $e(x)=e^{2\pi i x}$. The Dirichlet series above satisfy analytic continuation to the entire complex plane, which satisfy functional equations (see Section \ref{atsec}). The reciprocity law will now follow from the fact that the central value at $s=k/2$ of the additive twists define a {\it quantum modular form} of weight zero in the sense of Zagier \cite{Zagier10} (see Theorem \ref{quantumgeneral} for the precise statement).\\ 
The inspiration to this approach comes from Bettin's work \cite{Bettin16}, which gives an interpretation of the twisted first second moment (\ref{dirichlettwist}) in terms of the central value of the Estermann zeta function defined for $\Re s>1$ by
$$ D(a,c;s):= \sum_{n\geq 1} \frac{d(n)e(na/c)}{n^s}, $$
where $d(n)$ denotes the number of divisors of $n$. One can think of this as an additive twist of $\zeta(s)^2$, which in turn is the $L$-function associated to the $\GL_2$-object $\frac{\partial}{\partial s} E(z,s)_{s=1/2}$. Bettin's results can be interpreted as showing that $D(a/c;1/2):=D(a,c;1/2)$ defines a quantum modular form \cite[Theorem 1]{Bettin16}.\\
Recently Bettin and Drappeau \cite[Lemma 8.3]{BeDr19} showed quantum modularity in the case of level 1 cusp forms, which they ingeniously combined with dynamical methods to determine the limiting distribution of the central values of additive twists of $L$-functions of level 1 cusp forms. The results of this paper however extend the quantum modularity proved by Bettin and Drappeau to general discrete and co-finite subgroups of $\SL_2(\R)$ with a cusp at infinity (with an appropriate definition of quantum modularity). Quantum modularity for general levels will be needed if one wants to extend the methods of Bettin and Drappeau (see the remark on page 8 of \cite{BeDr19}).
\begin{remark}
A different proof of the limiting distribution of central values of additive twists was obtained by the author \cite{No19} using the theory of Eisenstein series twisted by modular symbols. The methods furthermore apply to cusp forms for congruence subgroups of general level and even general Fuchsian groups of the first kind with a cusp at $\infty$.
\end{remark}

Finally we discuss quantum modularity at $\infty$ and show that for weight 2 cusp forms, this is in a certain sense equivalent to the functional equation at the central point for the additive twists of the associated $L$-function. 
%Reciprocity laws for other families of $L$-functions were later discovered by Andersen and Kiral, Blomer, Li and Miller and Blomer and Khan building on work of Motohashi and Kutznezov.\\ 
%In particular \cite{AnKi} is concerned with a  $GL_2\times GL_2$ reciprocity law, which roughly speaking for a fixed cusp form $g$ of full level relates the two spectral averages
%$$ \sum_{f \text{ level }q} \omega_f L(g\otimes f, 1/2)^2 \lambda_f(p)   \rightsquigarrow  \sum_{f \text{ level }p} \omega_f L(g\otimes f, 1/2)^2 \lambda_f(q),  $$
%where $\omega_f$ is the sign of $f$ and $ \lambda_f(p)$ is the $p$th Hecke eigenvalue of $f$. In the last section we give a slight modification of a key step in their proof and note that the reciprocity in this case is also caused by (this time very simple) quantum modular forms constructed from Dirichlet character. The quantum modular forms which turn up are the maps $ \Q\rightarrow \C$ defined by
%$$  a/b \mapsto \Re (\chi(a)\overline{\chi(b)}), $$ 
%for $(a,b)=1$, where $\chi$ mod $c$ is a Dirichlet character. This is a quantum modular forms of weight $0$ for the subgroup of $GL_2(\R)$ generated by 
%$$\Gamma_1(c):=\{\gamma \in GL_2(\Z) \mid \gamma \equiv \begin{pmatrix} 1 & 0 \\ 0 & 1 \end{pmatrix}\mod c\}$$ 
%and $\begin{pmatrix} 0 & 1 \\ 1 & 0 \end{pmatrix}$. 

\section{Statement of results}%%%%%%%%%%%%%%%%%%
In order to state our results, let $f\in \mathcal{S}_k(\Gamma_0(N))$ be a new form of even weight $k$ and level $N$. Let $\omega_f$ denote the eigenvalue under the Fricke involution $W$ (see (\ref{fricke}) below) and let
$$ f(z)=\sum_{n\geq 1} a_f(n) q^n, \quad q=e^{2\pi i z},$$
be the Fourier expansion. Then given a Dirichlet character $\chi$, we consider the following twisted $L$-function; 
$$L(f\otimes \chi,s)= \sum_{n\geq 1} \frac{a_f(n)\chi(n)}{n^s},$$
normalized so that the central value is at $s=k/2$. Our result is the following reciprocity law for a certain appropriately weighted twisted first moment. 
\begin{thm} \label{recipr}
Let $f\in \mathcal{S}_k(\Gamma_0(N))$ be a new form of even weight $k$ and level $N$. Then we have for any pair of integers $0<l<q$ with $(q,Nl)=1$ that
\begin{align}
\label{mainreci}&\frac{1}{\varphi(q)} \sum_{\chi \text{ mod }q}\tau(\overline{\chi}^*)\nu(f, \chi^*,q/c(\chi)) L(f\otimes \chi^*,k/2)\chi(l)\\
\nonumber &  \qquad \qquad -\frac{\omega_f}{\varphi(lN)} \sum_{\chi \text{ mod }lN}\tau(\overline{\chi}^*)\nu(f, \chi^*,lN/c(\chi)) L(f\otimes \chi^*,k/2)\chi(-q)\\
\nonumber &\qquad \qquad \qquad \qquad\qquad \qquad \qquad \qquad =  L(f,k/2)+O_f(l/q),
\end{align}
where $\chi^*$ mod $c(\chi)$ denotes the conductor of the Dirichlet character $\chi$ and the arithmetic weights $\nu$ are given by
$$ \nu(f, \chi, n):= \sum_{\substack{n_1n_2n_3=n,\\ (n_1,q)=1}}\chi(n_1) \mu(n_1)\overline{\chi}(n_2) \mu(n_2) a_f(n_3)n_3^{1-k/2}.  $$
\end{thm}
\begin{remark}The reason for the specific shape of the arithmetic weights $\nu(f, \chi, n)$ is due to Lemma \ref{BirchSteven} below. Note that for primitive characters $\chi$, this factor is equal to 1.\end{remark}
In the simplest case where $q,l$ are prime and the level is 1, we get the following precise version of (\ref{reciprocity}).
\begin{cor}\label{cor1}
Let $f\in \mathcal{S}_k(\Gamma_0(1))$ be a cusp form of weight $k$ and level 1. Then we have for primes $ l<q$;
\begin{align*}
\frac{1}{\varphi^*(q)}\sum_{\substack{\chi \text{ mod }q\\ \chi \text{ primitive}}}  \tau(\overline{\chi})L(f\otimes \chi,k/2)\chi(l)-\frac{1}{\varphi^*(l)}\sum_{\substack{\chi \text{ mod }l\\ \chi \text{ primitive}}} \tau(\overline{\chi})L(f\otimes \chi,k/2) \chi(-q)\\
= L(f,k/2)+O_f(l/q+1/\sqrt{l}),
\end{align*}
where the sums are taken over all primitive characters mod $q$ and $l$ respectively and $\varphi^*(q)=q-2$ denotes the number of primitive characters mod $q$.
\end{cor}

\section{Twisted first moments and additive twists} %%%%%%%%%%%%%%%%%%
In this section we will introduce the additive twists of a cuspidal $L$-function and furthermore for congruence subgroups show a connection to the twisted first moments in (\ref{reciprocity}) using a formula due to Birch and Stevens.
\subsection{Additive twists}\label{atsec}%%%%%%%%%%%%%%%%%% 
We refer to \cite[Section 3.3]{No19} for a more detailed account.\\ 

Let $\Gamma$ be any discrete and co-finite subgroup of $\SL_2(\R)$ with a cusp at $\infty$ of width 1 (see \cite[Chapter 2]{Iw}) and consider a cusp form $f\in \mathcal{S}_k(\Gamma)$ of even weight $k$ with Fourier expansion (at $\infty$) given by
$$  f(z)=\sum_{n\geq 1} a_f(n)q^n. $$
Then for $r\in \R$ we define the {\it additive twist} of the $L$-function of $f$ as
\begin{align}\label{additivetwist} L(f\otimes e(r),s):= \sum_{n\geq 1} \frac{a_f(n)e(nr)}{n^s}, \end{align}
which converges absolutely for $\Re s>(k+1)/2$ by Hecke's bound;
\begin{align}\label{hecke}\sum_{n\leq X} |a_f(n)|^2\ll X^k.\end{align}
Thus the Dirichlet series (\ref{additivetwist}) defines a continuous function in $r$ for $\Re s>(k+1)/2$. \\
Furthermore if $r$ corresponds to a cusp of $\Gamma$, then $L(f\otimes e(r),s)$ admits analytic continuation to the entire complex plane, given by the following integral representation
\begin{align}\label{integralrep} L(f\otimes e(r),s) = \frac{(2\pi)^{s}}{ \Gamma(s)} \int_0^\infty f(r+iy) y^{s-1}dy,  \end{align}
well-defined for all $s\in \C$.\\
Finally if $r$ is in the $\Gamma$-orbit of $\infty$, say $r=\gamma \infty$, then we have the following functional equation; 
\begin{align}\nonumber \Lambda (f\otimes e(\gamma \infty),s)&:= \left(\frac{c}{2\pi}\right)^{s} \Gamma(s)L(f\otimes e(\gamma \infty),s)\\
\label{FE}&= (-1)^{k/2} \Lambda (f\otimes e(\gamma^{-1}\infty),k-s),   \end{align}
where $c$ is the lower-left entry of $\gamma$.
\subsection{The Birch--Stevens formula}%%%%%%%%%%%%%%%%%%
In the special case where $\Gamma=\Gamma_0(N)$, the set of cusps corresponds to $\Q\cup \{\infty\}$ and the classical Birch--Stevens formula \cite[Eq. 2.2]{Po14} expresses the central value $L(f\otimes \chi,k/2)$ for a primitive Dirichlet character $\chi$ mod $q$ in terms of additive twists;
\begin{align}  \tau(\overline{\chi}) L(f\otimes \chi,k/2) =\frac{1}{\varphi(q)} \sum_{a\in (\Z/q\Z)^\times} L(f\otimes e(a/q),k/2)\overline{\chi(a)}, \end{align}
where $\tau(\overline{\chi})$ denotes the Gauss-sum of $\overline{\chi}$.\\ 
For $k=2$ the central values $L(f\otimes e(a/q),1)$ are known as {\it modular symbols} and the above equality has been used for computations of the central values of $L$-functions of base-change of an elliptic curve over $\Q$. \\
If we furthermore assume that $f$ is a new form (in particular an eigenform for all Hecke operators), then we have the following generalization to non-primitive characters $\chi$.
\begin{lemma}
\label{BirchSteven}
Let $f\in \mathcal{S}_k(\Gamma_0(N))$ be a new form of even weight $k$ and $\chi$ a Dirichlet character mod $q$. Then we have
\begin{align}
\tau(\overline{\chi}^*)\nu(f, \chi^*,q/c(\chi)) L(f\otimes \chi^*,k/2)=\sum_{a\in (\Z/q\Z)^\times} \overline{\chi}(a )L(f\otimes e(a/q),k/2),
\end{align}
where $\chi^*$ mod $c(\chi)$ denotes the conductor of the Dirichlet character $\chi$ and
$$ \nu(f, \chi, n)= \sum_{\substack{n_1n_2n_3=n,\\ (n_1,q)=1}}\chi(n_1) \mu(n_1)\overline{\chi}(n_2) \mu(n_2) a_f(n_3)n_3^{1-k/2}.  $$
\end{lemma}
For a proof see \cite[Proposition 6.1]{No19}. 
\begin{remark}
This formula was also the essential ingredient for Bruggeman and Diamantis in \cite{BrDia16}, where they give an explicit formula for the constant Fourier coefficient of Eisenstein series twisted by modular symbols.
\end{remark} 
From the above formula we conclude, using orthogonality of characters on the finite group $ (\Z/c\Z)^\times$, the following identity.
\begin{cor}\label{Additivetwist}
Let $f\in \mathcal{S}_k(\Gamma_0(N))$ be a new form of even weight $k$. Then we have
\begin{align}
L(f\otimes e(a/q),k/2)=  \frac{1}{\varphi(q)}\sum_{\chi \text{ mod }q}      \tau(\overline{\chi}^*)\nu(f, \chi^*,c/c(\chi)) L(f\otimes \chi^*,k/2) \chi(a),
\end{align}
with $\chi^*,c(\chi)$ and $\nu$ as above.
\end{cor}
Thus for $l=o(q)$, it follows from (\ref{mainreci}) that
$$ L(f\otimes e(l/q),k/2)-L(f\otimes e(-q/(Nl)),k/2)= L(f,k/2)+o(1),  $$
as $q\rightarrow \infty$. This is in sharp contrast with the average behavior since we know from the distribution result \cite[Theorem 1.4]{No19} that for $q$ fixed, the number 
$$L(f\otimes e(l/q),k/2),$$ 
is typically of magnitude $\sqrt{\log q}$. %there is how ever no contradiction since {l/q : l=o(q)} is a measure zero set

\section{Quantum modularity of additive twists} %%%%%%%%%%%%%%%%%%
The notion of {\it quantum modular forms} was introduced by Zagier in \cite{Zagier10} with one of the first examples appearing in earlier work with Lawrence \cite{LawZa99} on certain symmetries of quantum invariants of 3-knots. In Zagier's original definition quantum modular forms are maps $\mathbb{P}^1(\Q)\backslash S\rightarrow \C$ with $S$ a finite set which satisfy a variation of the modular transformation rule for congruence subgroups $\Gamma_0(N)\subset \SL_2(\Z)$ acting on $\mathbb{P}^1(\Q)$. One should think of the equivalence classes of $\mathbb{P}^1(\Q)$ under the action of $\Gamma_0(N)$ as the boundary of the symmetric space $\Gamma_0(N) \backslash\H$.\\ 
In this paper we study quantum modularity for general co-finite, discrete subgroups $\Gamma\subset \SL_2(\R)$ with a cusp at $\infty$. Although quantum modularity for general Fuchsian groups of the first kind has not been explicitly defined before in the literature, the definition is implicit in the introduction of \cite{Zagier10}. To make this definition, denote by $s(\Gamma)\subset \R\cup \{\infty\}$ the set of cusps of $\Gamma$ (i.e. fixed points of parabolic motions of $\Gamma$). In particular if $\Gamma=\Gamma_0(N)$, we have $s(\Gamma)=\Q\cup \{\infty\}=\mathbb{P}^1(\Q)$. Then we have the following definition. %, which is the natural interpretation of the comment by Zagier in the beginning of \cite{Zagier10}.

\begin{defi}\label{qmoddef}Let $\Gamma$ be a discrete, co-finite subgroup of ${\rm SL}_2(\R)$ with a cusp at $\infty$ of width 1 and consider a map $f:s(\Gamma)\backslash S\rightarrow \C$ with $S$ a finite set. Then $f$ is a {\bf quantum modular form} of weight $k$ for $\Gamma$ if the following holds; for all $\gamma\in\Gamma$ the function $g_\gamma(r):s(\Gamma)\backslash (S\cup \gamma^{-1}S)\rightarrow \C$ defined by
\begin{align}\label{quantum} g_\gamma(r):=f(\gamma r)-j(\gamma,r)^k f(r), \end{align}
extends to a \underline{continuous} function $\mathbb{P}^1(\R)\backslash (S\cup \gamma^{-1}S) \rightarrow \C $. \end{defi}
%Here continuity at $\infty$ is interpreted in the standard way. 
\begin{remark}\label{contex}Note that we have a large class of uninteresting examples coming from restrictions of continuous maps $\mathbb{P}^1(\R)\backslash S\rightarrow \C$ with $S$ a finite set.\end{remark}
\begin{remark} Here "continuous" can be replaced by different notions of regularity ($\mathcal{C}^1$, smooth, analytic, $\ldots$). \end{remark}
%-Connection to Eichler integrals(for integral weight Eichler integrals are not interesting examples),-strong quantum modular forms, -Hecke operators-(WHAT ABOUT $r\rightarrow -\infty$???)

\subsection{Proof of quantum modularity}
In this section we present a proof of the quantum modularity for the central values of additive twists of cuspidal $L$-functions. The proof uses the integral representation of the additive twist (\ref{integralrep}) and is similar in spirit to the treatment of Eichler integrals of half-integral cusp forms by Bringmann and Rolen in \cite{BriRol16}. One can also consider the Eichler integrals of an integral weight $k$ cusp form $f$, which corresponds to the special value $L(f\otimes e(r), k-1)$ of the additive twists. For $k>2$ this is however not an interesting example in our context since this is the restriction of a continuous function\footnote{Quantum modularity of Eichler integrals of integral weight cusp forms were studied by Lee in \cite{Lee18}. In this work the notion of regularity used in the definition of quantum modular forms is {\it smooth} instead of continuous. Now one gets a non-trivial result since the Eichler integrals do not define a smooth function, but the discrepancy (\ref{quantum}) is even a polynomial in this case.} (because of the absolute convergence of the Dirichlet series (\ref{additivetwist}) at the special value $s=k-1$) as was also noted in \cite[Section 1.4.1]{BeDr19}. \\
Using \cite[Theorem 1.4]{No19} one can see that the the central value $s=k/2$ of the additive twists of the $L$-function of an integral weight cusp form $f$ considered as a function of the twisting parameter $r\in s(\Gamma)$, is not the restriction of a continuous function. Our result is that this does however define a quantum modular form.
\begin{thm}[Quantum modularity]\label{quantumgeneral}
Let $\Gamma$ be a discrete, co-finite subgroup of $\SL_2(\R)$ with a cusp at $\infty$ of width 1 and let $f\in \mathcal{S}_k(\Gamma)$ be a cusp form of even weight $k$. Then the map $s(\Gamma)\backslash\{\infty\} \rightarrow \C$ defined by
$$ r\mapsto L(f\otimes e(r),k/2) $$
is a quantum modular form of weight zero for $\Gamma$. More precisely for $\gamma\in \Gamma$ and $r\in s(\Gamma)\backslash\{\infty\}$ with $\gamma r\neq \infty$ and $\gamma \infty\neq \infty$, we have
\begin{align}
\nonumber L(f\otimes e(\gamma r),k/2)-&L(f\otimes e(r),k/2)\\
\nonumber =L(f\otimes e(\gamma \infty),k/2)&+\sum_{j=1}^{k/2-1} \frac{\binom{k/2-1}{j}}{(c^{-1}j(\gamma, r))^j}\frac{(-2\pi i)^{-j}\Gamma(k/2+j)}{\Gamma(k/2)}L(f\otimes e(r), k/2+j)\\
\label{exactformula1}&+\sum_{j=1}^{k/2-1} \frac{\binom{k/2-1}{j}}{(cj(\gamma,r))^j}\frac{(-2\pi i)^j\Gamma(k/2-j)}{\Gamma(k/2)}L(f\otimes e(\gamma \infty),k/2-j),
\end{align}
where $c$ is the lower-left entry of $\gamma$.
\end{thm} 
\begin{proof}
We have to show continuity of the discrepancy (\ref{quantum}) for all $\gamma\in\Gamma$. First of all if $\gamma \infty=\infty $ then it is easy to see that 
$$L(f\otimes e(\gamma r),k/2)=L(f\otimes e(r),k/2),$$
for all $r\in s(\Gamma)\backslash \{\infty\}$ since $f$ is 1-periodic. Thus quantum modularity for such $\gamma$ is clear.\\ 
Recall that by (\ref{hecke}), the additive twists $L(f\otimes e(x), k/2+j)$ for $j>0$ define continuous functions in $x\in \R$. Thus for $\gamma$ fixed the two sums in (\ref{exactformula1}) both extend to continuous functions $\R\backslash \{\gamma^{-1}\infty\}\rightarrow \C$. This shows that it suffices to prove (\ref{exactformula1}).\\
To prove (\ref{exactformula1}) we begin with the following integral representation
\begin{align*}  L(f\otimes e(\gamma r),k/2) &= \frac{(2\pi)^{k/2}}{\Gamma(k/2)} \int_{0}^{\infty} f(\gamma r+iy) y^{k/2-1}dy\\
&= \frac{(-2\pi i)^{k/2}}{\Gamma(k/2)} \int_{\gamma r}^{i\infty} f(z) (z-\gamma r)^{k/2-1}dz, \end{align*}
where the integral is taken along the vertical line from $\gamma r$ to $i\infty$. Now the integrand is holomorphic and we can apply Cauchy's theorem to write
$$L(f\otimes e(\gamma r),k/2) =  \frac{(-2\pi i)^{k/2}}{\Gamma(k/2)} \left(   \int_{\gamma r}^{\gamma \infty} f(z) (z-\gamma r)^{k/2-1}dz+ \int_{\gamma \infty}^{\infty} f(z) (z-\gamma r)^{k/2-1}dz  \right).$$
We will now treat the two integrals separately. In the first integral we do the change of variable $z\mapsto \gamma z$ and use the fact that 
$$ \gamma z-\gamma r= \frac{z-r}{j(\gamma,z)j(\gamma, r)}, $$
to arrive at
\begin{align*}
 &\int_{\gamma r}^{\gamma \infty} f(z) (z-\gamma r)^{k/2-1}dz\\
 &=  \int_{r}^{ i\infty} f(\gamma z) (\gamma z-\gamma r)^{k/2-1}\frac{dz}{j(\gamma , z)^2}\\
 &=  \left(\frac{c}{j(\gamma, r)}\right)^{k/2-1}\int_{r}^{ i\infty} f(z) \left((z-r)(z-r+\frac{j(\gamma,r)}{c})\right)^{k/2-1}dz\\
 &= \sum_{j=0}^{k/2-1} \binom{k/2-1}{j}\left( \frac{c}{j(\gamma, r)}\right)^j \int_r^{i\infty} f(z) (z-r)^{k/2+j-1}dz\\
 &= \sum_{j=0}^{k/2-1} \binom{k/2-1}{j}\left( \frac{c}{j(\gamma, r)}\right)^j \frac{\Gamma(k/2+j)}{(-2\pi i)^{k/2+j}}L(f\otimes e(r),k/2+j).
\end{align*}
A similar treatment of the other integral gives
\begin{align*}
 &\int^{i\infty}_{\gamma \infty} f(z) (z-\gamma \infty +(\gamma \infty-\gamma r))^{k/2-1}dz\\
 &= \sum_{j=0}^{k/2-1} \binom{k/2-1}{j}\left( \frac{1}{cj(\gamma, r)}\right)^{j} \int^{i\infty}_{\gamma \infty} f(z) (z-\gamma \infty)^{k/2-j-1}dz\\
&= \sum_{j=0}^{k/2-1} \binom{k/2-1}{j}\left( \frac{1}{cj(\gamma, r)}\right)^{j} \frac{\Gamma(k/2-j)}{(-2\pi i)^{k/2-j}} L(f\otimes e(\gamma \infty),k/2-j),
\end{align*}
which finishes the proof.
\end{proof}
\begin{remark} For $k=2$ we observe that the right-hand side of (\ref{exactformula1}) is just a constant. From this it follows immediately that the central value of the additive twists (i.e. modular symbols) define a {\it strong quantum modular form} in the sense of Zagier \cite{Zagier10}.
\end{remark}
\begin{remark}
The proof and statement of Theorem \ref{quantumgeneral} is very similar to \cite[Lemma 5.1]{No19}, which is used to reduce the study of Eisenstein series twisted by (generalized) modular symbols to the study of certain "completions" in the sense of \cite[page 6]{BrDia16}. This was used to determine the distribution of central values of additive twists in \cite{No19} and thus there seem to be some similarities with the methods in \cite{BeDr19}, which would be interesting to understand better.
\end{remark}
If $\Gamma=\Gamma_0(N)$ is a congruence group and we assume that $f\in \mathcal{S}_k(\Gamma_0(N))$ is a new form, we get a similar result for the Fricke involution, defined as
\begin{align}\label{fricke} Wf(z):= N^{-k/2}z^{-k} f(H_Nz),\end{align}
where
$$  H_N=\begin{pmatrix} 0& -1 \\ N & 0  \end{pmatrix}.$$
The proof is essentially the same and we will omit it.
\begin{thm}\label{quantumhecke}
Let $f\in \mathcal{S}_k(\Gamma_0(N))$ be a cuspidal new form of even weight $k$ and level $N$. Then we have for all $r\in \Q\backslash \{0\}$ that
\begin{align}
\nonumber &L(f\otimes e(-1/(Nr)),k/2)-L(f\otimes e(r),k/2)\\
\nonumber &=L(f,k/2)+\omega_f\sum_{j=1}^{k/2-1} \frac{\binom{k/2-1}{j}}{r^j}\frac{(-2\pi i)^{-j}\Gamma(k/2+j)}{\Gamma(k/2)}L(f\otimes e(r), k/2+j)\\
\label{exactformula2}&\qquad \qquad\qquad+\sum_{j=1}^{k/2-1} \frac{\binom{k/2-1}{j}}{(N r)^j}\frac{(-2\pi i)^j\Gamma(k/2-j)}{\Gamma(k/2)}L(f,k/2-j),
\end{align}
where $\omega_f=\pm 1$ is the eigenvalue of $f$ under the Fricke involution $W$.
\end{thm} 

\subsection{Quantum modularity at $\infty$ and the functional equation} %%%%%%%%%%%%%%%
A natural question to ask is whether we can extend our quantum modular forms to $\infty$. This is equivalent to whether we can assign a value at $\infty$ such that the right-hand side of (\ref{exactformula1}) converges, as respectively $r\rightarrow \infty$ and $ r\rightarrow \gamma^{-1}\infty$, to the left-hand side of (\ref{exactformula1}) with respectively $r=\infty$ and $ r= \gamma^{-1}\infty$. In the special case when $k=2$ this can be done by putting  $ L(f\otimes e(\infty),1)=0 $. It turns out that quantum modularity at $\infty$ amounts to the functional equation for $L(f\otimes e(\gamma \infty),s)$ at the central point.
\begin{thm}
Let $\Gamma$ be a discrete, co-finite subgroup of $SL_2(\R)$ with a cusp at $\infty$ of width 1 and let $f\in \mathcal{S}_2(\Gamma)$. Then the map $s(\Gamma) \rightarrow \C$ defined by
$$ r\mapsto L(f\otimes e(r),1) $$
where we put $ L(f\otimes e(\infty),1)=0 $, is a quantum modular form of weight zero for $\Gamma$.
\end{thm}
\begin{proof}
First of all if $\gamma \infty=\infty$ then the result is clear for all $k$ even. Furthermore when $k=2$, (\ref{exactformula1}) reduces to
\begin{align}\label{infty}L(f\otimes e(\gamma r),1)-L(f\otimes e(r),1)=L(f\otimes e(\gamma \infty),1) ,\end{align}
and since the right-hand side is constant we can ignore convergence completely. \\
Using Theorem \ref{quantumgeneral} we only need to check that (\ref{infty}) still holds at $r=\infty$ and $r=\gamma^{-1}\infty$. The first case is immediate and the second case reduces to
$$ L(f\otimes e(\gamma^{-1}\infty),1)\overset{?}{=}-L(f\otimes e(\gamma \infty),1), $$
which is exactly the functional equation (\ref{FE}) at the central point $s=k/2=1$. This finishes the proof.
\end{proof}
\begin{remark}This seems to be a very special phenomena for $k=2$ and the author was not able to prove anything for higher weights. We even have numerical data suggesting that it is not true for the Ramanujan $\Delta$-function. It would be interesting to investigate quantum modularity at $\infty$ for other quantum modular forms in the litterature.\end{remark}
%In order to generalize this discussion to $k>2$ it is enough to prove that the derivative for $L(f\otimes e(r),k/2+1)$ exists at $\gamma^{-1}\infty$, i.e.
%$$  \lim_{r\rightarrow \gamma^{-1}\infty} \frac{L(f\otimes e(r),k/2+1)-L(f\otimes e(\gamma^{-1}\infty),k/2+1)}{r-\gamma^{-1}\infty}=L(f\otimes e(\gamma^{-1}\infty),k/2),   $$
%but it is not clear to the author whether this is correct or not.

\section{From quantum modularity to reciprocity laws} %%%%%%%%%%%%%%%%%%%%%%%%%%%%%%%%%
In this final section, we will prove the reciprocity laws Theorem \ref{recipr} and Corollary \ref{cor1} using the quantum modularity for the Fricke involution of additive twists proved above.
\subsection{Proof of Theorem \ref{recipr} and Corollary \ref{cor1}} %%%%%%%%%%%%%%%%%%%%%%%%%%%%%%
By combining Corollary \ref{Additivetwist} and Theorem \ref{quantumhecke} (with $r=-q/(lN)$) we get an explicit formula for the left-hand side of (\ref{mainreci}). The expression for the error term on the right-hand side of (\ref{mainreci}) follows by the following estimate of the right-hand side of (\ref{exactformula2});
\begin{align*}
&L(f,k/2)+\omega_f\sum_{j=1}^{k/2-1} \frac{\binom{k/2-1}{j}}{r^j}\frac{(-2\pi i)^{-j}\Gamma(k/2+j)}{\Gamma(k/2)}L(f\otimes e(r), k/2+j)\\
&+\sum_{j=1}^{k/2-1} \frac{\binom{k/2-1}{j}}{(N^2 r)^j}\frac{(-2\pi i)^j\Gamma(k/2-j)}{\Gamma(k/2)}L(f,k/2-j)\\
&=L(f,k/2)+O_{f}\left( \sum_{j=1}^{k/2-1}  r^{-j} \right) \\
&=L(f,k/2)+ O_{f}(r^{-1}),
\end{align*}
where we again used the following uniform bound 
$$|L(f\otimes e(r), k/2+j)|\leq \sum_{n\geq 1}\frac{|a_f(n)|}{n^{k/2+1}}<\infty,$$
for $j\geq 1$. This proves Theorem \ref{recipr}. \\
Furthermore if $l$ and $q$ are both primes then the only non-primitve character modulo $l$ and $q$ are the principal characters. Now using Deligne's bound $a_f(n)\ll d(n)n^{(k-1)/2}$ on the Fourier coefficients of $f$, we derive  
$$ \tau(\overline{\chi_0}^*)\nu(f, \chi_0^*,q/c(\chi_0))=\nu(f, 1,q)=a_f(q)q^{1-k/2}-1 \ll_f q^{1/2}, $$ 
and similarly for $l$. Using that 
$$ \frac{1}{\varphi(q)}=\frac{1}{\varphi^*(q)}+O(q^{-2}), $$
we conclude the proof of Corollary \ref{cor1}.
\bibliography{/Users/Soemanden/Documents/MatematikSamlet/Bib-tex/mybib}
\bibliographystyle{amsplain}
\end{document}